\documentclass[12pt]{amsart}

\usepackage{enumerate, amsmath, amsthm, amsfonts, amssymb, xy,  mathrsfs, graphicx, paralist, fancyvrb}
\usepackage[usenames, dvipsnames]{xcolor}
\usepackage[margin=1in]{geometry} 
\usepackage[bookmarks, colorlinks=true, linkcolor=blue, citecolor=blue, urlcolor=blue]{hyperref}

\input xy
\xyoption{all}

\numberwithin{equation}{subsection}
\newtheorem{theorem}[equation]{Theorem}

\newtheorem{proposition}[equation]{Proposition}

\theoremstyle{definition}
\newtheorem{rmk}[equation]{Remark}
\newenvironment{remark}[1][]{\begin{rmk}[#1] \pushQED{\qed}}{\popQED \end{rmk}}
\newtheorem{eg}[equation]{Example}
\newenvironment{example}[1][]{\begin{eg}[#1] \pushQED{\qed}}{\popQED \end{eg}}
\newtheorem{defn}[equation]{Definition}

\makeatletter
\@addtoreset{equation}{section}
\renewcommand{\thesubsection}{%
  \ifnum\c@subsection<1 \@arabic\c@section
  \else \thesection.\@arabic\c@subsection
  \fi
}
\makeatother

\newcommand{\bC}{\mathbf{C}}

\newcommand{\cE}{\mathcal{E}}

\newcommand{\bF}{\mathbf{F}}

\newcommand{\bO}{\mathbf{O}}

\newcommand{\cQ}{\mathcal{Q}}

\newcommand{\cR}{\mathcal{R}}

\newcommand{\bS}{\mathbf{S}}

\newcommand{\fS}{\mathfrak{S}}

\newcommand{\bZ}{\mathbf{Z}}

\newcommand{\fg}{\mathfrak{g}}

\renewcommand{\phi}{\varphi}

\newcommand{\ol}[1]{\overline{#1}}

\newcommand{\arxiv}[1]{\href{http://arxiv.org/abs/#1}{{\tt arXiv:#1}}}

\makeatletter
\def\Ddots{\mathinner{\mkern1mu\raise\p@
\vbox{\kern7\p@\hbox{.}}\mkern2mu
\raise4\p@\hbox{.}\mkern2mu\raise7\p@\hbox{.}\mkern1mu}}
\makeatother

\renewcommand{\hom}{\operatorname{Hom}}

\DeclareMathOperator{\trace}{Tr}

\DeclareMathOperator{\End}{End}

\DeclareMathOperator{\Sym}{Sym}

\DeclareMathOperator{\Spec}{Spec}

\DeclareMathOperator{\sgn}{sgn}

\newcommand{\GL}{\mathbf{GL}}

\newcommand{\SO}{\mathbf{SO}}
\newcommand{\Gr}{\mathbf{Gr}}

\newcommand{\fgl}{\mathfrak{gl}}
\newcommand{\fso}{\mathfrak{so}}
\newcommand{\fsp}{\mathfrak{sp}}

\newcommand{\fosp}{\mathfrak{osp}}

\title[Jacobi--Trudi formulas and determinantal varieties]{Jacobi--Trudi formulas and\\ determinantal varieties}

\author{Steven V Sam}
\address{Department of Mathematics, University of California, San Diego, CA USA}
\email{ssam@ucsd.edu}
\thanks{SS was supported by NSF DMS-1849173.}

\author{Jerzy Weyman}
\address{Department of Mathematics, Jagiellonian University, Krak\'ow, Poland}
\email{jerzy.weyman@gmail.com}
\thanks{JW was supported by MAESTRO NCN-UMO-2019/34/A/ST1/00263 and NAWA POWROTy -PPN/PPo/2018/1/00013/U/00001 grants, as well as by NSF DMS-1802067.}

\date{June 10, 2022}

\begin{document}

\maketitle

\begin{abstract}
  Gessel gave a determinantal expression for certain sums of Schur functions which visually looks like the classical Jacobi--Trudi formula. We explain the commonality of these formulas using a construction of Zelevinsky involving BGG complexes and use this explanation to generalize this formula in a few different directions.
\end{abstract}

\section{Introduction}

In this paper we attempt to understand and generalize some results of Gessel \cite{gessel} which bear some visual similarities to the classical Jacobi--Trudi formulas in symmetric function theory. First, we recall the statement: given an integer partition $\lambda$ with at most $k$ parts, we have the determinantal formula for the Schur function
\[
  s_\lambda = \det( h_{\lambda_i - i + j} )_{i,j=1,\dots,k}
\]
where $h$ denotes the complete homogeneous symmetric function. The representation-theoretic significance of this formula is that it supplies a recipe for constructing the Schur functor from tensor products of symmetric power functors. Namely, the expansion of the above determinant has a natural interpretation as the Euler characteristic of an acyclic chain complex. Gessel's formula replaces $s_\lambda$ with the sum
\[
  \sum_{\substack{\lambda\\ \ell(\lambda)\le k}} s_\lambda(x) s_\lambda(y)
\]
in two sets of variables $x$ and $y$, and replaces $h_n$ with the sum
\[
  H_n = \sum_{d \ge 0} h_d(x) h_{d+n}(y).
\]
The significance of this formula for us, and the starting point of this paper, is that the first expression is the character of the coordinate ring of the determinantal variety of generic matrices of rank $\le k$, while the $H_n$ are characters of certain equivariant modules supported on the variety of rank $\le 1$ matrices. Naturally, we want to interpret this formula as a recipe for constructing the variety of rank $\le k$ matrices from these more basic modules.

In this article, we reprove this formula using representation theory, and in particular, interpret it as the Euler characteristic of an acyclic chain complex. This chain complex is quite interesting and it would be worthwhile to further investigate them from a geometric perspective. Moreover, from our approach we deduce several similar formulas: one for the character of the coordinate ring of determinantal variety of skew-symmetric matrices of rank $\le 2k$, one for a ring very closely related to the coordinate ring of the determinantal of symmetric matrices of rank $\le k$, and a companion formula for the symmetric case involving spinor representations.

Our approach involves an old construction of Zelevinsky \cite{zelevinsky} involving BGG complexes. He used this method to construct an acyclic complex of $\GL_n$-representations whose Euler characteristic gives the Jacobi--Trudi formula for Schur polynomials. Zelevinsky's approach takes as input a representation $V$ of a semisimple Lie algebra $\fg$ and a weight $\lambda$, and outputs an acyclic complex whose terms are certain weight spaces of $V$ and which resolves the space of highest weight vectors of weight $\lambda$ in $V$. If $V$ carries an action of another algebra $H$ which commutes with $\fg$, then the resulting complex is also compatible with the $H$-action.

We apply Zelevinsky's result to various infinite-dimensional representations arising from Howe dual pairs. In many of the cases the weight $\lambda$ we use is the trivial weight. However, we will work in the context of a general weight since the formulas are quite similar, and hence we will get quite a vast generalization of the original formula.

The paper is organized as follows. In Section~\ref{sec:zelevinsky} we recall Zelevinsky's formula. In Section~\ref{sec:generic-mat} we deal with generic matrices. To get Gessel's result on the coordinate ring of determinantal varieties we apply Zelevinsky's formula to the space
\[
  V = \Sym (E\otimes U^*)\otimes \Sym (U\otimes F^*)
\]
where $\dim(U)=k$ and $\fg=\fgl (U)$. The commuting action is the action of $H=\fgl (E)\times \fgl (F)$.

In Sections~\ref{sec:skew-sym} and \ref{sec:sym} we deal with skew-symmetric and symmetric matrices, respectively. We use the space
\[
  V = \Sym(E\otimes U)
\]
where $U$ is equipped with either a symplectic or orthogonal form. The original action is that of $\fg=\fsp (U)$ or $\fg = \fso(U)$ and the commuting action is that of $H =\fgl(E)$.

Finally in Section~\ref{sec:spinor} we apply Zelevinsky's result to the space
\[
  V = \Sym(E\otimes U)\otimes\Delta
\]
where again $U$ is equipped with an orthogonal form, and $\Delta$ is the spinor representation. This allows us to deduce determinantal formulas for the sums
\[
  \sum_{\substack{\lambda\\ \ell(\lambda) \le k}} s_\lambda
\]
which recover some formulas from \cite{gessel}.

For all of the above results, one can use exterior algebras in place of symmetric algebras. However, this does not give anything essentially new because of the existence of the involution $\omega$ on symmetric functions that sends $s_\lambda$ to its transpose $s_{\lambda^\dagger}$. We briefly remark on this in Remark~\ref{rmk:exterior} in the first case and don't discuss it any further.

\subsection*{Notation}

We use $\Sym^d$ to denote the $d$th symmetric power functor and $\Sym = \bigoplus_{d \ge 0} \Sym^d$ to denote the symmetric algebra construction. Similarly, $\bigwedge^d$ is the $d$th exterior power and $\bigwedge = \bigoplus_{d \ge 0} \bigwedge^d$ is the exterior algebra construction.

For symmetric function notation, we follow \cite[Chapter 7]{stanley} except that the transpose of a partition is denoted with $\dagger$ rather than a prime.

\subsection*{Related work}

Determinantal expressions for variations of these sums obtained by restricting representations of orthogonal and symplectic Lie algebras to the general linear Lie algebra can be obtained from \cite{krattenthaler}. Jacobi--Trudi formulas can also be used to give formulas for minimal affinizations in the study of representations of loop algebras, see \cite{minaff} and the references there.

\subsection*{Acknowledgements}

We thank Christian Krattenthaler and Claudiu Raicu for helpful discussions.

\section{The setup} \label{sec:zelevinsky}

Let $\fg$ be a reductive complex Lie algebra. We will assume that we have fixed the data of a Cartan subalgebra and set of positive roots. Let $W$ be its Weyl group and let $\rho$ be $\frac12$ times the sum of all of the positive roots. Let $V$ be a locally finite  $\fg$-representation (i.e., $V$ is isomorphic to a direct sum of finite-dimensional $\fg$-representations). Given a dominant weight $\lambda$, let $V[\lambda]$ be the space of highest weight vectors of weight $\lambda$ in $V$ and given any weight $\chi$, let $V_\chi$ be the $\chi$-weight space.

The following theorem of Zelevinsky \cite{zelevinsky} is crucial for this article:

\begin{theorem}[Zelevinsky]
  There is a (finite) exact sequence
  \[
    \cdots \to  \bF_1 \to \bF_0 \to V[\lambda] \to 0
  \]
  where
  \[
    \bF_i = \bigoplus_{\substack{w \in W\\ \ell(w)=i}} V_{\lambda+\rho - w^{-1}(\rho)}.
  \]
\end{theorem}

We will denote this complex either by $\bF^\lambda_\bullet$ or $\bF_\bullet$ depending on the context.

We note that Zelevinsky uses $w(\rho)$ rather than $w^{-1}(\rho)$, but this does not affect the statement since $\ell(w)=\ell(w^{-1})$. We use this modification to simplify some notation.

We will be interested in the case when an algebra $H$ acts on $V$ so that it commutes with $\fg$. Then $V[\lambda]$ is an $H$-module and the complex $\bF_\bullet$ is $H$-equivariant. In all of our cases of interest, $H$ is a reductive Lie algebra. The equivariant Euler characteristic of $\bF_\bullet$ equals the character of $V[\lambda]$, and we will interpret it as a determinant.

\section{Generic matrices} \label{sec:generic-mat}

Let $E,F,U$ be finite-dimensional vector spaces with
\[
  \dim(E)=e, \qquad \dim(F)=f, \qquad \dim U = k,
\]
and set
\[
  V = \Sym(E \otimes U^*) \otimes \Sym(U \otimes F^*).
\]
We set $\fg = \fgl(U)$. There is a commuting action of $H = \fgl(E) \times \fgl(F)$ on $V$.

\begin{remark}
In fact, we get a commuting action of a larger Lie algebra $H' = \fgl(E \oplus F)$ so that $V$ is a direct sum of irreducible $\fg \times H'$ representations (for the explicit formulas for the action, see \cite[\S 5.6.6, Exercise 1]{goodman-wallach}).  
\end{remark}

The $\fgl(U)$-equivariant inclusion $E \otimes F^* \subset (E\otimes U^*) \otimes (U \otimes F^*)$ extends to an algebra homomorphism
\[
  \Sym(E \otimes F^*) \to V.
\]
It is well-known that the image of this map is $V^{\fgl(U)}$, the space of $\fgl(U)$-invariants, and if we interpret $E \otimes F^*$ as the linear functions on $\hom(E,F)$, then the kernel is the ideal generated by the minors of size $k+1$ \cite[Theorems 5.2.1, 12.2.12]{goodman-wallach}.

We identify weights with $k$-tuples of complex numbers and under this identification we can take
\[
  \rho = (k,k-1,\dots,1).
\]
We remark that shifting this choice of $\rho$ by any multiple of $(1,1,\dots,1)$ will not affect any of the formulas below, so we merely make this particular choice for convenience.

If $\dim U = 1$, then for each integer $n$, the $n$-weight space of $V$ is
\[
  L_n = \bigoplus_{d \ge 0} \Sym^d(E) \otimes \Sym^{d+n}(F^*).
\]

\begin{proposition}
For general $k$, we have
\[
  V_\chi = L_{\chi_1} \otimes \cdots \otimes L_{\chi_k}.
\]
\end{proposition}

\begin{proof}
  Pick a weight space decomposition $U = U_1 \oplus \cdots \oplus U_k$. Then we have
  \[
    V \cong \bigotimes_{i=1}^k (\Sym(E \otimes U_i^*)\otimes \Sym(U_i \otimes F^*)).
  \]
  Then the $\chi$-weight space of $V$ is the tensor product over $i$ of the $\chi_i$-weight space of $\Sym(E \otimes U_i^*) \otimes \Sym(U_i \otimes F^*)$.
\end{proof}

\begin{remark}
  $L_n$ is an irreducible representation of $\fgl(E \oplus F)$, though the restriction of this action to $\fgl(E) \times \fgl(F)$ must be modified so that it is the usual action twisted by the character $(A,B) \mapsto \frac12 \trace(A) - \frac12 \trace(B)$.
\end{remark}

Now we consider the general setup. If $X$ is a representation of $\fg$, we use $[X]$ as notation for its character.

\begin{proposition}
  Given a decreasing sequence $\lambda \in \bZ^k$, the character of $V[\lambda]$ is 
\[
  \det( [L_{\lambda_i -i + j}] )_{i,j =1,\dots, k}.
\]
\end{proposition}

\begin{proof}
  We have $W = \fS_k$ and for $w \in W$, $w(\rho)_i = k+1-w^{-1}(i)$ and $(-1)^{\ell(w)} = \sgn(w)$, so the equivariant Euler characteristic of $\bF^\lambda_\bullet$ is given by
  \[
    \sum_{w \in \fS_k} \sgn(w) [L_{\lambda_1 - 1 + w(1)}] \cdots [L_{\lambda_k -k +w(k)}].
  \]
  which is the Laplace expansion of the claimed determinantal expression.
\end{proof}

When $\lambda=0$, this can be used to recover the formula of Gessel in \cite[Theorem 16]{gessel}, which was stated in the language of symmetric functions. To be precise, $V[0] = V^{\fgl(U)}$ is the coordinate ring of the variety of rank $\le k$ matrices of size $e \times f$. Its character is a polynomial in $x_1,\dots,x_e,y_1,\dots,y_f$ which is separately symmetric in the $x$ variables and the $y$ variables and has the expression
\[
  \sum_{\substack{\lambda\\ \ell(\lambda) \le k}} s_\lambda(x_1,\dots,x_e) s_\lambda(y_1,\dots,y_f)
\]
where $s_\lambda$ denotes the Schur polynomial indexed by $\lambda$ and the sum is over all partitions with at most $k$ rows. Letting $h_n = s_{(n)}$, the result above says that this sum is given by the determinant
\[
  \det \left( \sum_{d \ge 0} h_{d}(x_1,\dots,x_e) h_{d+j-i}(y_1,\dots,y_f) \right)_{i,j=1,\dots,k}.
\]
The specialization of a Schur function $s_\lambda$ to $n$ variables is zero if and only if $n<\ell(\lambda)$. Hence if we take $e,f \ge k$, then all of the Schur polynomials above can be replaced by Schur functions in countably many variables, and we get precisely the claimed formula from \cite[Theorem 16]{gessel}. This discussion of the difference between symmetric polynomials in finitely many variables and symmetric functions in infinitely many variables is equally applicable in all later cases, so we won't make any further comment on it.

\begin{remark}
  Actually, Zelevinsky \cite{zelevinsky} gives a more general result that utilizes two weights $\lambda,\mu$. Using this more general formula, the skew Jacobi--Trudi determinant
  \[
    \det([L_{\lambda_i - \mu_j - i + j}])_{i,j=1,\dots,k}
  \]
  computes the character of $\sum_\nu V[\nu]^{\oplus c^\lambda_{\mu, \nu}}$. We don't know if this representation carries any significance. This applies to all cases to follow, but we won't make any further mention of it.
\end{remark}

\begin{example}
  Consider the case $\dim U = 2$. Then Zelevinsky's theorem gives a complex
\[
  0 \to L_1 \otimes L_{-1} \to L_0 \otimes L_0
\]
which ``resolves'' the coordinate ring of rank $\le 2$ matrices.
\end{example}

\begin{remark}
  We can express a highest weight $\lambda$ as a pair $(\mu,\mu')$ where $\ell(\mu)+\ell(\mu') \le \dim U$ and this means $(\mu,0,\dots,0,-{\mu'}^{\rm op})$. The module $V[\lambda]$ is an irreducible $H'$-representation; this is $B_{\mu,\mu'} = M_{\mu,\mu'}$ in \cite[\S 5.5]{lwood}, where it is shown to have the following geometric construction (see \cite{weyman} for general information on this type of construction). If $\ell(\mu) \le a$ and $\ell(\mu') \le b$, define $X = \Gr(e-a,E) \times \Gr(f-b,F^*)$ and  consider the trivial bundle $\cE = (E^* \otimes F) \times X$. Let $\cR_1 \subset E \times X$ denote the pullback of the tautological subbundle on $\Gr(e-a,E)$ and similarly define $\cR_2$. Also define $\cQ_1 = E/\cR_1$ and $\cQ_2 = F^*/\cR_2$. Then $\xi = \cR_1 \otimes \cR_2$ gives linear equations for a subbundle $\Spec(\Sym(\eta))$ where
  \[
    \eta = (E \otimes F^*) / \xi.
  \]
  Let $\pi \colon \cE \to E^* \otimes F$ denote the projection. Then we have
  \[
    V[\lambda] = \pi_*(\bS_\mu(\cQ_1) \otimes \bS_{\mu'}(\cQ_2) \otimes \Sym(\eta))
  \]
  and in fact the higher direct images vanish. We do not know of a simple formula for its character which isn't an alternating sum.
\end{remark}

\begin{remark} \label{rmk:exterior}
  We could instead use the representation
  \[
    V = \bigwedge(E \otimes U^*) \otimes \bigwedge(U \otimes F^*).
  \]
  However, this doesn't give anything substantially new: on the level of characters, it just amounts to applying the $\omega$ involution to the previous case for both $\fgl(E)$ and $\fgl(F)$.

  The same remark applies to the representation
  \[
    V = \bigwedge(E \otimes U^*) \otimes \Sym(U \otimes F^*).
  \]
  However, in this case, the commutator of $\fgl(U)$ is the Lie superalgebra $\fgl(E|F)$, so we get some interesting determinantal expressions for the characters for a certain class of its representations.
\end{remark}

\section{Skew-symmetric matrices} \label{sec:skew-sym}

Let $E$ be a finite-dimensional vector space with $\dim(E)=e$ and let $U$ be a $2k$-dimensional symplectic space and set
\[
  V = \Sym(E \otimes U).
\]
We set $\fg = \fsp(U)$. There is a commuting action of $H = \fgl(E)$ on $V$.

\begin{remark}
  There is a natural orthogonal form on $E \oplus E^*$ and there is a commuting action of the larger Lie algebra $H' = \fso(E \oplus E^*)$ so that $V$ is a direct sum of irreducible $\fg \times H'$ representations (for the explicit formulas for the action, see \cite[\S 5.6.5]{goodman-wallach}).
\end{remark}

The symplectic form on $U$ gives a $\fsp(U)$-equivariant inclusion
\[
  \bigwedge^2E \subset \bigwedge^2E \otimes \bigwedge^2U \subset \Sym^2(E\otimes U)
\]
(the second inclusion is via the space of $2 \times 2$ determinants) which extends to an algebra homomorphism
\[
  \Sym(\bigwedge^2 E) \to V.
\]
It is well-known that the image of this map is $V^{\fsp(U)}$, the space of $\fsp(U)$-invariants, and if we interpret $\bigwedge^2 E$ as the linear functions on the space of skew-symmetric matrices $\bigwedge^2(E^*)$, then the kernel is the ideal generated by the Pfaffians of size $2(k+1)$ \cite[Theorems 5.2.2, 12.2.15]{goodman-wallach}.

We identify weights of $\fsp(U)$ with $k$-tuples of complex numbers and under this identification, we have
\[
  \rho = (k,k-1,\dots,1).
\]

If $\dim U = 2$, then for each integer $n$, the $n$-weight space of $V$ is
\[
  L_n = \bigoplus_{d \ge 0} \Sym^d(E) \otimes \Sym^{d+n}(E).
\]

\begin{proposition}
For general $k$, we have
\[
  V_\chi = L_{\chi_1} \otimes \cdots \otimes L_{\chi_k}.
\]
\end{proposition}

\begin{proof}
  Pick a weight space decomposition $U = (U_1 \oplus U_1^*) \oplus \cdots \oplus (U_k \oplus U_k^*)$. Then we have
  \[
    V \cong \bigotimes_{i=1}^k (\Sym(E \otimes (U_i \oplus U_i^*))).
  \]
  Then the $\chi$-weight space of $V$ is the tensor product over $i$ of the $\chi_i$-weight space of $\Sym(E \otimes (U_i \oplus U_i^*))$.
\end{proof}

\begin{remark}
  $L_n$ is an irreducible representation of $\fso(E \oplus E^*)$, though the restriction of this action to $\fgl(E)$ must be modified so that it is the usual action twisted by the character $A \mapsto -\trace(A)$. 
\end{remark}

Now we consider the general setup. If $X$ is a representation of $\fg$, we use $[X]$ as notation for its character.

\begin{proposition}
  Given a partition $\lambda$ with $\ell(\lambda)\le k$, the character of $V[\lambda]$ is 
\[
  \det ([L_{\lambda_i -i+j}] - [L_{\lambda_i -i+ 2k+2-j}])_{i,j=1,\dots,k}.
\]
\end{proposition}

\begin{proof}
  Every element of $W$ can be factored as $\alpha w$ where $\alpha \in \{\pm 1\}^k$ and $w \in \fS_k$; since negation in the last entry is a Coxeter generator and all negations in a single entry are conjugate, we get $\ell(\alpha w) = |\alpha| + \ell(w) \pmod 2$ where $|\alpha|$ is the number of negative signs of $\alpha$. So the Euler characteristic of the complex $\bF^\lambda_\bullet$ is
\begin{align*}
  &  \sum_{\alpha\in \{\pm 1\}^k} (-1)^{|\alpha|} \sum_{w \in \fS_k} \sgn(w) [L_{\lambda_1 + k - \alpha_1 (k+1-w(1))}] \cdots [ L_{\lambda_k + 1 - \alpha_k (k+1-w(k))}] \\
  =& \sum_{w \in \fS_k} \sgn(w) ([L_{\lambda_1 - 1 + w(1)}] - [L_{\lambda_1 + 2k+1-w(1)}]) \cdots ([L_{\lambda_k - k + w(k)}] - [L_{\lambda_k + k+2-w(k)}])\\
  =& \det ([L_{\lambda_i -i+j}] - [L_{\lambda_i -i+ 2k+2-j}])_{i,j=1,\dots,k}. \qedhere
\end{align*}
\end{proof}

This gives an analogue of Gessel's determinantal formula for the coordinate ring of skew-symmetric matrices of rank $\le 2k$ by taking $\lambda=0$ and applying the substitution $i\mapsto k+1-i$ and $j\mapsto k+1-j$, which we record as the following theorem.

\begin{theorem} For each $k$, we have
  \[
  \sum_{\substack{\lambda\\ \lambda_1 \le k}} s_{(2\lambda)^\dagger} = \det ([L_{j-i}] - [L_{i+j}])_{i,j=1,\dots,k}.
\]
\end{theorem}

\begin{remark}
The module $V[\lambda]$ is an irreducible $H'$-representation; this is $B_{\lambda} = M_{\lambda}$ in \cite[\S 3.5]{lwood}, where it is shown to have the following geometric construction. Define $X = \Gr(e-k,E^*)$ and consider the trivial bundle $\cE = \bigwedge^2 E^* \times X$. Let $\cR \subset E^* \times X$ denote the  tautological subbundle on $\Gr(e-k,E^*)$ and define $\cQ = E^*/\cR$. Then $\xi = \bigwedge^2 \cR$ gives linear equations for a subbundle $\Spec(\Sym(\eta))$ where
  \[
    \eta = \bigwedge^2 E^* / \xi.
  \]
  Let $\pi \colon \cE \to \bigwedge^2 E^*$ denote the projection. Then we have
  \[
    V[\lambda] = \pi_*(\bS_\lambda\cQ \otimes \Sym(\eta))
  \]
  and in fact the higher direct images vanish. We do not know of a simple formula for its character which isn't an alternating sum.
\end{remark}

\begin{example}
  For $\dim U = 2$, Zelevinsky's theorem gives a complex
\[
  0 \to L_{-2} \to L_0
\]
which ``resolves'' the coordinate ring of rank $\le 2$ skew-symmetric matrices.
\end{example}

\begin{example}
  If $\dim U = 4$, we get
  \[
    0 \to L_{4} \otimes L_{2} \to
    {\begin{array}{c} L_{4} \otimes L_0 \\ L_{3} \otimes L_{3} \end{array}} \to
    {\begin{array}{c} L_{3} \otimes L_1 \\ L_{1} \otimes L_{3} \end{array}} \to
    {\begin{array}{c} L_{0} \otimes L_{2} \\ L_{1} \otimes L_{-1} \end{array}} \to
    L_0 \otimes L_0
  \]
  which ``resolves'' the coordinate ring of the rank $\le 4$ skew-symmetric matrices.
\end{example}

\section{Symmetric matrices} \label{sec:sym}

Let $E$ be a finite-dimensional vector space with $\dim(E)=e$ and let $U$ be an $m$-dimensional orthogonal space and set
\[
  V = \Sym(E \otimes U).
\]
We set $\fg = \fso(U)$. There is a commuting action of $H = \fgl(E)$ on $V$.

\begin{remark}
  There is a natural symplectic form on $E \oplus E^*$ and there is a commuting action of the larger Lie algebra $H' = \fsp(E \oplus E^*)$ so that $V$ is a direct sum of irreducible $\bO(U) \times H'$ representations (for the explicit formulas for the action, see \cite[\S 5.6.3]{goodman-wallach}).
\end{remark}

The orthogonal form on $U$ gives a $\bO(U)$-equivariant inclusion
\[
  \Sym^2E \subset \Sym^2E \otimes \Sym^2U \subset \Sym^2(E\otimes U)
\]
which extends to an algebra homomorphism
\[
  \Sym(\Sym^2 E) \to V.
\]
It is well-known that the image of this map is $V^{\bO(U)}$, the space of $\bO(U)$-invariants, and if we interpret $\Sym^2 E$ as the linear functions on the space of symmetric matrices $\Sym^2(E^*)$, then the kernel is the ideal generated by the minors of size $m+1$ \cite[Theorems 5.2.2, 12.2.14]{goodman-wallach}.

There is a subtle difference when compared to the previous cases: the invariants for the group $\bO(U)$ and the subgroup $\SO(U)$ (or equivalently, the Lie algebra $\fso(U)$) are not the same. In fact, the invariant space $V^{\fso(U)}$ is a degree 2 module over the determinantal ring $V^{\bO(U)}$. All of our results will be about the action of $\fso(U)$.

We will treat the cases of $m$ odd and $m$ even separately.

\subsection{Even case}

First suppose that $m=2k$ is even. We identify weights of $\fso(U)$ with $k$-tuples of complex numbers. Then
\[
  \rho = (k-1,k-2,\dots,0).
\]

First consider the case $\dim U = 2$. For each integer $n$, the $n$-weight space of $V$ is
\[
  L_n = \bigoplus_{d \ge 0} \Sym^d(E) \otimes \Sym^{d+n}(E).
\]

\begin{proposition} \label{prop:even-orth}
For general $k$, we have
\[
  V_\chi = L_{\chi_1} \otimes \cdots \otimes L_{\chi_k}.
\]
\end{proposition}

\begin{proof}
  Pick a weight space decomposition $U = (U_1 \oplus U_1^*) \oplus \cdots \oplus (U_k \oplus U_k^*)$. Then we have
  \[
    V \cong \bigotimes_{i=1}^k (\Sym(E \otimes (U_i \oplus U_i^*))).
  \]
  Then the $\chi$-weight space of $V$ is the tensor product over $i$ of the $\chi_i$-weight space of $\Sym(E \otimes (U_i \oplus U_i^*))$.
\end{proof}

\begin{remark}
  For $n \ne 0$, $L_n$ is an irreducible representation of $\fsp(E \oplus E^*)$, though the restriction of this action to $\fgl(E)$ must be modified so that it is the usual action twisted by the character $A \mapsto -\trace(A)$. For $n=0$, $L_0$ is a direct sum of two irreducible representations which can be described as
  \[
    \bigoplus_{d \ge 0} \Sym^2(\Sym^d E), \qquad \bigoplus_{d \ge 0} \bigwedge^2(\Sym^d E). \qedhere
  \]
\end{remark}

\begin{proposition}
  Given a partition $\lambda$ with $\ell(\lambda)\le k$, the character of $V[\lambda]$ is 
\[
  \frac12 \det ([L_{\lambda_i -i+j}] + [L_{\lambda_i -i+ 2k-j}])_{i,j=1,\dots,k}.
\]
\end{proposition}

\begin{proof}
Let $\{\pm 1\}^k_0$ be the subgroup of $\{\pm 1\}^k$ consisting of elements with an even number of $-1$'s. Every element of the Weyl group can be factored as $\alpha w$ where $\alpha \in (\bZ/2)^k_0$ and $w \in \fS_k$; since negating the last 2 entries and swapping them is a Coxeter generator, we get $\ell(\alpha w) = \ell(w) \pmod 2$. So the Euler characteristic of the complex $\bF^\lambda_\bullet$ is
\begin{align*}
  &  \sum_{\alpha \in \{\pm 1\}^k_0} \sum_{w \in \fS_k} \sgn(w) [L_{\lambda_1 + k - 1 - \alpha_1 (k - w(1))}] \cdots [ L_{\lambda_k - \alpha_k (k - w(k))}]
\end{align*}
Given $w \in \fS_k$, let $i = w^{-1}(k)$. Then for either choice of $\alpha_i \in \{\pm 1\}$, we get the same term $[L_{\lambda_i + k - i}]$. In particular, we may sum over all choices of $\alpha \in \{\pm 1\}^k$ if we divide by 2:
\begin{align*}
&  \frac12 \sum_{\alpha \in \{\pm 1\}^k} \sum_{w \in \fS_k} \sgn(w) [L_{\lambda_1 + k - 1 - \alpha_1 (k - w(1))}] \cdots [ L_{\lambda_k - \alpha_k (k - w(k))}]\\
  =& \frac12 \sum_{\sigma \in \fS_k} \sgn(w) ([L_{\lambda_1 -1 + w(1)}] + [L_{\lambda_1 + 2k - 1 - w(1)}]) \cdots ([L_{\lambda_k - k + w(k)}] + [L_{\lambda_k + k - w(k)}])\\
  =&  \frac12 \det ([L_{\lambda_i - i + j}] + [L_{\lambda_i -i + 2k - j}])_{i,j = 1, \dots, k}. \qedhere
\end{align*}
\end{proof}

\begin{remark}
  The module $V[\lambda]$ is $\ol{B}_{\lambda} = \ol{M}_{\lambda}$ in \cite[\S 4.6]{lwood}, where it is shown to have the following geometric construction. Define $X = \Gr(e-k,E^*)$ and consider the trivial bundle $\cE = \Sym^2 E^* \times X$. Let $\cR \subset E^* \times X$ denote the tautological subbundle on $\Gr(e-k,E^*)$ and define $\cQ = E^*/\cR$. Then $\xi = \Sym^2 \cR$ gives linear equations for a subbundle $\Spec(\Sym(\eta))$ where
  \[
    \eta = \Sym^2 E^* / \xi.
  \]
  Let $\pi \colon \cE \to \Sym^2 E^*$ denote the projection. Then we have
  \[
    V[\lambda] = \pi_*(\bS_\lambda\cQ \otimes \Sym(\eta))
  \]
  and in fact the higher direct images vanish. Note that $\Spec(V[0])$ is a double cover of a determinantal variety and that each $V[\lambda]$ is in fact supported on it. We do not know of a simple formula for its character which isn't an alternating sum.
\end{remark}

The $\fso(U)$ representation $\bS_\nu(U)$ has nonzero invariants if and only if, writing $\nu = (\nu_1,\dots,\nu_{2k})$, we have that all $\nu_i$ are even, or all $\nu_i$ are odd. Furthermore, when this holds, the space of $\fso(U)$-invariants is always 1-dimensional (this follows from \cite[\S 11.2.1, Theorem]{procesi}). Hence, when $\lambda=0$, we get the following special case of the previous result:
\[
  \sum_{\substack{\mu\\ \ell(\mu) \le 2k}} (s_{(1+2\mu_1,\dots,1+2\mu_{2k})} + s_{2\mu}) = \frac12 \det([L_{j-i}] + [L_{2k-i-j}])_{i,j=1,\dots,k}.
\]

\subsection{Odd case}

Suppose $m=2k+1$ is odd. Then
\[
  \rho = (k-\frac12,k-\frac32,\dots,\frac12).
\]

\begin{proposition}\label{prop:odd-orth}
For general $k$, we have
\[
  V_\chi = \Sym(E) \otimes L_{\chi_1} \otimes \cdots \otimes L_{\chi_k}.
\]
\end{proposition}

\begin{proof}
  Pick a weight space decomposition $U = \bC \oplus (U_1 \oplus U_1^*) \oplus \cdots \oplus (U_k \oplus U_k^*)$. Then we have
  \[
    V \cong \Sym(E) \otimes \bigotimes_{i=1}^k (\Sym(E \otimes (U_i \oplus U_i^*))).
  \]
  Then the $\chi$-weight space of $V$ is the tensor product over $i$ of the $\chi_i$-weight space of $\Sym(E \otimes (U_i \oplus U_i^*))$.
\end{proof}

\begin{proposition}
  Given a partition $\lambda$ with $\ell(\lambda)\le k$, the character of $V[\lambda]$ is 
  \[
    [\Sym(E)] \det ([L_{\lambda_i-i+j}] - [L_{\lambda_i + 2k + 1 - i - j}])_{i,j=1,\dots, k}.
  \]
\end{proposition}

\begin{proof}
Every element of the Weyl group can be factored as $\alpha w$ where $\alpha \in \{\pm 1\}^k$ and $w \in \fS_k$; since negation in the last entry is a Coxeter generator and all negations in a single entry are conjugate, we get $\ell(\alpha w) = |\alpha| + \ell(w) \pmod 2$ where $|\alpha|$ is the number of $-1$ in $\alpha$. So the Euler characteristic of this complex becomes
\begin{align*}
  &  \sum_{\alpha\in \{\pm 1\}^k} (-1)^{|\alpha|} \sum_{w \in \fS_k} \sgn(w) [\Sym(E)] [L_{\lambda_1 + k -\frac12 - \alpha_1 (k + \frac12 - w(1))}] \cdots [ L_{\lambda_k + \frac12 - \alpha_k (k + \frac12 - w(k))}] \\
  =& [\Sym(E)] \sum_{w \in \fS_k} \sgn(w) ([L_{\lambda_1 - 1 + w(1)}] - [L_{\lambda_1 + 2k - w(1)}]) \cdots ([L_{\lambda_k - k + w(k)}] - [L_{\lambda_k + k + 1 - w(k)}])\\
  =& [\Sym(E)] \det ([L_{\lambda_i-i+j}] - [L_{\lambda_i + 2k + 1 - i - j}])_{i,j=1,\dots, k}. \qedhere
\end{align*}
\end{proof}

As in the previous case with $m$ even, when $\lambda=0$, we get the following special case of the previous result:
\[
  \sum_{\substack{\mu\\ \ell(\mu) \le m}} (s_{(1+2\mu_1,\dots,1+2\mu_{m})} + s_{2\mu}) = [\Sym(E)] \det([L_{j-i}] + [L_{2k+1-i-j}])_{i,j=1,\dots,k}.
\]

Since $m$ is odd, the element $-1$ is in the center of $\bO(U)$ and hence acts on the complex $\bF_\bullet$ and each $V[\lambda]$, so we can further refine it by taking isotypic components of the corresponding $\bZ/2$-action. Let $V[\lambda]^+$ denote the space of invariants under $-1$ and $V[\lambda]^-$ denote the space of skew-invariants. Define
\begin{align*}
  M_0 &= \bigoplus_{d \ge 0} \Sym^{2d} E, \qquad  M_1 = \bigoplus_{d \ge 0} \Sym^{2d+1}E,
\end{align*}
where the indices $0,1$ are to be thought of as elements of $\bZ/2$.

\begin{proposition}
  The character of $V[\lambda]^+$ is given by
  \begin{align*}
    \sum_{\alpha\in \{\pm 1\}^k} (-1)^{|\alpha|} \sum_{w \in \fS_k} \sgn(w) [M_{|\lambda|+|\alpha|}] [L_{\lambda_1 + k -\frac12 - \alpha_1 (k + \frac12 - w(1))}] \cdots [ L_{\lambda_k + \frac12 - \alpha_k (k + \frac12 - w(k))}],
  \end{align*}
  and the character of $V[\lambda]^-$ is given by
  \[
    \sum_{\alpha\in \{\pm 1\}^k} (-1)^{|\alpha|} \sum_{w \in \fS_k} \sgn(w) [M_{|\lambda|+|\alpha|+1}] [L_{\lambda_1 + k -\frac12 - \alpha_1 (k + \frac12 - w(1))}] \cdots [ L_{\lambda_k + \frac12 - \alpha_k (k + \frac12 - w(k))}].
  \]
\end{proposition}

We unfortunately could not find a compact determinantal expression for the above sums.

\begin{proof}
The element $-1$ acts on $\Sym^d(E)$ by $(-1)^d$, so it acts on $L_n$ by $(-1)^n$ and hence
\begin{align*}
  V_\chi^+ = \begin{cases} M_0 \otimes L_{\chi_1} \otimes \cdots \otimes L_{\chi_k} & \text{if $|\chi|=\chi_1+\cdots+\chi_k$ is even}\\
    M_1 \otimes L_{\chi_1} \otimes \cdots \otimes L_{\chi_k} & \text{otherwise} \end{cases},
\end{align*}
and the opposite holds for $V_\chi^-$. Also, $|\lambda+\rho-\alpha w(\rho)| \equiv |\lambda| + |\alpha| \pmod 2$ for any $\alpha \in \{\pm 1\}^k$ and $w \in \fS_k$, so the result follows.
\end{proof}

\begin{example}
  For $\dim U = 3$, Zelevinsky's theorem gives a complex
\[
  0 \to \Sym E \otimes L_{-1} \to \Sym E \otimes L_0
\]
which ``resolves'' the double cover of the coordinate ring of rank $\le 3$ symmetric matrices. Taking invariants under $\bZ/2$ gives the resolution for the coordinate ring itself:
\[
  0 \to M_1 \otimes L_{-1} \to M_0 \otimes L_0. \qedhere
\]
\end{example}

\begin{remark}
For the pushforward construction, see \cite[\S 4.7]{lwood}.
\end{remark}

\section{``Non-commutative'' matrices via spinors} \label{sec:spinor}

Let $E$ be a finite-dimensional vector space with $\dim(E)=e$ and let $U$ be an $m$-dimensional orthogonal space, let $\Delta$ be the spinor representation (this is irreducible if $m$ is odd and is the direct sum of both half-spinor representations if $m$ is even) and set
\[
  V = \Sym(E \otimes U) \otimes \Delta.
\]
We set $\fg = \fso(U)$. There is a commuting action of $H = \fgl(E)$.

\begin{remark}
There is a natural symplectic form on $E \oplus E^*$ and a symmetric form on $\bC$, and there is a commuting action of the orthosymplectic Lie superalgebra $H' = \fosp(\bC | E \oplus E^*)$ so that $V$ is a direct sum of irreducible ${\bf Pin}(U) \times H'$ representations where ${\bf Pin}(U)$ denotes the natural double cover of the orthogonal group $\bO(U)$.
\end{remark}

The $\Delta$-covariants for the action of ${\bf Pin}(U)$ is the quotient of $\Sym(E) \otimes \bigwedge(\Sym^2 E)$ which is like a non-commutative version of the coordinate ring of a determinantal variety. This particular vector space appears because it is the underlying space of the universal enveloping algebra of the free 2-step nilpotent Lie superalgebra $E \oplus \Sym^2E$ generated by $E$, see \cite[\S 4.1]{spincat}. More specifically,
\[
   \Sym(E) \otimes \bigwedge(\Sym^2 E) \cong\bigoplus_\lambda \bS_\lambda(E)
\]
and we take the quotient by all $\bS_\lambda(E)$ with $\ell(\lambda)>m$. 

Again define
\[
  L_n = \bigoplus_{d \ge 0} \Sym^d(E) \otimes \Sym^{d+n}(E).
\]

We will treat the cases of $m$ odd and $m$ even separately.

Let $I_k$ be the set of 0-1 vectors of length $k$. Recall that the set of weights for $\Delta$ are $v-(\frac12, \dots, \frac12)$ for $v\in I_k$ if $k=\lfloor m/2 \rfloor$.

\subsection{Odd case}

Suppose $m=2k+1$ is odd. Then
\[
  \rho = (k-\frac12,k-\frac32,\dots,\frac12).
\]

\begin{proposition}
\[
  V_\chi = \Sym(E) \otimes \bigoplus_{v \in I_k} L_{\chi_1+v_1-\frac12} \otimes \cdots \otimes L_{\chi_k+v_k-\frac12}.
\]
\end{proposition}

\begin{proof}
  This is similar to the proof of Proposition~\ref{prop:odd-orth} except we need to also take into account the weight space decomposition of $\Delta$.
\end{proof}

\begin{proposition}
  If $\lambda- (\frac12, \dots,\frac12)$ is a partition, then the character of $V[\lambda]$ is
  \[
    [\Sym(E)] \det ([L_{\lambda_{k+1-i} - \frac12 - i + j}] - [L_{\lambda_{k+1-i} - \frac12 + i + j}])_{i,j =1,\dots, k}.
  \]
\end{proposition}

\begin{proof}
  Every element of the Weyl group can be factored as $\alpha w$ where $\alpha \in \{\pm 1\}^k$ and $w \in \fS_k$; since negation in the last entry is a Coxeter generator and all negations in a single entry are conjugate, we get $\ell(\alpha w) = |\alpha| + \ell(w) \pmod 2$ where $|\alpha|$ is the number of $-1$ in $\alpha$. So the Euler characteristic of this complex becomes
\begin{align*}
  &  \sum_{\alpha \in \{\pm 1\}^k} (-1)^{|\alpha|} \sum_{w \in \fS_k} \sgn(w) [\Sym(E)] \sum_{x \in I_k} [L_{\lambda_1 + k - 1 - \alpha_1 (k + \frac12 - w(1)) + x_1}] \cdots [ L_{\lambda_k - \alpha_k (k + \frac12 - w(k)) + x_k}] \\
  =& [\Sym(E)] \sum_{w \in \fS_k} \sgn(w) \prod_{i=1}^k ([L_{\lambda_i - i - \frac12 + w(i)}] - [L_{\lambda_i + 2k-i + \frac12 - w(i)}] + [L_{\lambda_i - i + \frac12 + w(i)}] - [L_{\lambda_i + 2k-i + \frac32 - w(i)}])\\
  =& [\Sym(E)] \det ([L_{\lambda_i - i - \frac12 + j}] - [L_{\lambda_i + 2k - i + \frac12 - j}] + [L_{\lambda_i - i + \frac12 + j}] - [L_{\lambda_i + 2k-i + \frac32 - j}])_{i,j =1,\dots, k}. 
\end{align*}
We can further simplify as follows. First note that if $j=k$, then the middle two terms cancel. In general, the inner two terms for any $j<k$ match the outer two terms for $j+1$, so subtract column $j+1$ from column $j$ (starting with $j=k-1$ and decreasing the index). Then we get
\begin{align*}
   [\Sym(E)] \det ([L_{\lambda_i - i - \frac12 + j}] - [L_{\lambda_i + 2k-i + \frac32 - j}])_{i,j =1,\dots, k}.
\end{align*}
Now do the change of indices $i\mapsto k+1-i$ and $j \mapsto k+1-i$ to get
\[
    [\Sym(E)] \det ([L_{\lambda_{k+1-i} - \frac12 - i + j}] - [L_{\lambda_{k+1-i} - \frac12 + i + j}])_{i,j =1,\dots, k}. \qedhere
\]
\end{proof}

\begin{remark}
  For the pushforward construction, see \cite[\S 6.3]{spincat}.
\end{remark}

When $\lambda=(\frac12, \dots, \frac12)$, then $V[\lambda] = \bigoplus_{\ell(\lambda) \le m} \bS_\lambda(E)$ (see \cite[Proposition 4.1]{spincat}), and its character is given by the determinant
\[
  [\Sym(E)] \det([L_{j-i}] - [L_{i+j}])_{i,j =1,\dots, k}.
\]
This is the formula given in \cite[Theorem 14, Equation 22]{gessel} once we note that $L_n=L_{-n}$.

\subsection{Even case}

Now suppose that $m=2k$ is even. Then
\[
  \rho = (k-1,k-2,\dots,0).
\]

\begin{proposition}
\[
  V_\chi =\bigoplus_{v \in I_k} L_{\chi_1 + v_1 - \frac12} \otimes \cdots \otimes L_{\chi_k + v_k - \frac12}.
\]
\end{proposition}

\begin{proof}
  This is similar to the proof of Proposition~\ref{prop:even-orth} except we need to also take into account the weight space decomposition of $\Delta$.  
\end{proof}

\begin{proposition}
  If $\mu = \lambda- (\frac12, \dots,\frac12)$ is a weakly decreasing integer sequence and satisfies $\mu_{k-1} \ge |\mu_k|$, then the character of $V[\lambda]$ is
  \[
    \det ([L_{\lambda_{k+1-i} - i + j - \frac12}] + [L_{\lambda_{k+1-i} + i  + j -  \frac32}])_{i,j =1,\dots, k}.
  \]
\end{proposition}

\begin{proof}
  Let $\{\pm 1\}^k_0$ be the subgroup of $\{\pm 1\}^k$ with an even number of $-1$. Every element of the Weyl group can be factored as $\alpha w$ where $\alpha \in \{\pm 1\}^k_0$ and $w \in \fS_k$; since negating the last 2 entries and swapping them is a Coxeter generator, we get $\ell(\alpha w) = \ell(w) \pmod 2$. So the Euler characteristic of this complex becomes
\begin{align*}
  &  \sum_{\alpha \in \{\pm 1\}^k_0} \sum_{w \in \fS_k} \sgn(w) \sum_{x \in I_k} [L_{\lambda_1 + k - 1 - \alpha_1 (k - w(1)) + x_1 - \frac12}] \cdots [ L_{\lambda_k - \alpha_k (k - w(k)) + x_k - \frac12}]\\
  = &\frac12\sum_{\alpha \in \{\pm 1\}^k} \sum_{w \in \fS_k} \sgn(w) \sum_{x \in I_k} [L_{\lambda_1 + k - 1 - \alpha_1 (k - w(1)) + x_1 - \frac12}] \cdots [ L_{\lambda_k - \alpha_k (k - w(k)) + x_k - \frac12}]\\
  =& \frac12 \sum_{\sigma \in \fS_k} \sgn(\sigma) \prod_{i=1}^k ([L_{\lambda_i - i + w(i) - \frac12}]
+ [L_{\lambda_i + 2k - i - w(i)- \frac12}] + [L_{\lambda_i - i + w(i) + \frac12}] + [L_{\lambda_i + 2k - i - w(i) + \frac12}])\\
  =&  \frac12 \det ([L_{\lambda_i - i + j - \frac12}] + [L_{\lambda_i -i - j + 2k-\frac12}] + [L_{\lambda_i - i + j + \frac12}] + [L_{\lambda_i - i -j + 2k+ \frac12}])_{i,j =1,\dots, k}
\end{align*}
where the first equality follows from the fact that if $\sigma(i)=k$, then either choice of $\alpha_i \in \{\pm 1\}$ gives the same result. We can simplify this determinant further. If $j=k$, then the first two terms agree and the last two terms agree, so we can pull out a factor of 2. For $j<k$, the inner two terms match the outer two terms for $j+1$, so we can do column operations to get the following result:
\[
  \det ([L_{\lambda_i - i + j - \frac12}] + [L_{\lambda_i - i -j + 2k+ \frac12}])_{i,j =1,\dots, k}.
\]
Finally, do the change of indices $i\mapsto k+1-i$ and $j\mapsto k+1-j$ to get
\[
  \det ([L_{\lambda_{k+1-i} - i + j - \frac12}] + [L_{\lambda_{k+1-i} + i  + j -  \frac32}])_{i,j =1,\dots, k}. \qedhere
\]
\end{proof}

\begin{remark}
  For the pushforward construction, see \cite[\S 6.2]{spincat}.
\end{remark}

If we take the $\Delta$-coinvariants of $V$ (as a representation of ${\bf Pin}(U)$), we get $\bigoplus_{\ell(\lambda) \le m} \bS_\lambda(E)$ (see \cite[Proposition 4.1]{spincat}). Since we're working with the Lie algebra $\fso(U)$, we can instead take the sum of the coinvariants for the two weights $\lambda^{\pm} =(\frac12, \dots, \frac12, \pm \frac12)$ to get twice the desired representation (since $\dim \End_{\fso(U)}(\Delta)=2$). In particular, we get the following formula for the character of $\bigoplus_{\ell(\lambda) \le 2k} \bS_\lambda(E)$:
\[
  \det ([L_{j- i}] + [L_{i  + j - 1}])_{i,j =1,\dots, k}.
\]
Again, this agrees with \cite[Theorem 14]{gessel}.

\end{document}